\documentclass[12pt,reqno]{amsart}
\usepackage{graphicx, cite, hyperref}
\usepackage{amsfonts}
\usepackage{enumerate}
\usepackage[caption = false]{subfig}
\usepackage{varioref}
\numberwithin{equation}{section}

 \theoremstyle{plain}

\newtheorem{thm}{Theorem}[section]

\newtheorem{lem}[thm]{Lemma}

\theoremstyle{definition}
\newtheorem{defn}[thm]{Definition}
\theoremstyle{remark}

\makeatother

\newcommand{\de}{{\mathbb D}}
\newcommand{\ce}{{\mathbb C}}

\begin{document}

\title[]{Briot-Bouquet differential subordinations of analytic functions
	involving the Mittag-Leffler function defined in Cardioid domain}

\author[Asena \c{C}etinkaya, \c{S}ahsene Alt\i nkaya]{Asena \c{C}etinkaya$^\ast$, \c{S}ahsene Alt\i nkaya}
\address{Department of Mathematics and Computer Science\\
\.{I}stanbul K\"{u}lt\"{u}r University, \.{I}stanbul, T\"{u}rkiye}
\email{asnfigen@hotmail.com}

\address{Department of Mathematics and Statistics,
	University of Turku, FI-20014 Turku, 
	Finland}
\email{sahsenealtinkaya@gmail.com}

\begin{abstract}
In this researh work, we establish a new subclass of analytic functions constructed by the Mittag-Leffler function that maps the open unit disc onto the region bounded by the Cardioid domain. Using  a technique introduced by Miller and Mocanu, we investigate several  Briot–Bouquet  differential subordinations  for this function class.
\end{abstract}

\keywords{Mittag-Leffler function, Briot-Bouquet differential subordination, Briot-Bouquet differential equation, Cardioid domain.\\
 $^\ast$Corresponding Author}

\maketitle

\section{Introduction}
Denote by $\mathcal{A}$ the class of analytic functions $f$ of the form 
\begin{equation}
	f(z)=z+\sum_{n=1}^{\infty }a_{n+1}z^{n+1}  \label{aa}
\end{equation}
in the open unit disc ${\mathbb{D}}:=\{z\in {\mathbb{C}}:|z|<1\}$. Denote by $\Omega $ the class of Schwarz functions $\vartheta $ which are analytic in ${\mathbb{D}}$ with $\vartheta (0)=0$ and $|\vartheta
(z)|<1$. For analytic functions $f_{1}$ and $f_{2}$ in ${\mathbb{D}}$, we state that $f_{1}$ is subordinate to $f_{2}$, symbolized by $f_{1}\prec f_{2}$, if there exists a function $\vartheta$ in $\Omega $ fulfilling $f_{1}(z)=f_{2}(\vartheta (z))$. The
comprehensive details of subordination can be found in \cite{Duren}.

The convolution of functions $f_{1}(z)=z+\sum_{n=1}^{\infty }a_{n+1}z^{n+1}$ and $f_{2}(z)=z+\sum_{n=1}^{\infty }b_{n+1}z^{n+1}$ is expressed by 
\begin{eqnarray*}
	f_{1}(z)\ast f_{2}(z)=(f_{1}\ast f_{2})(z)=z+\sum_{n=1}^{\infty}a_{n+1}b_{n+1}z^{n+1}, \ \ \ (z\in {\mathbb{D}}).
\end{eqnarray*}

Ma and Minda \cite{minda94} investigated the class of analytic functions $h$ with  positive real part in ${\mathbb{D}}$ that map the disc ${\mathbb{D}}$ onto regions starlike with respect to $1$, symmetric with respect to the real axis and normalized by the conditions $h(0)=1$ and $h^{\prime }(0)>0$, and introduced the class
\begin{eqnarray*} 
	\mathcal{S}^\ast(h)=\left\{ f\in \mathcal{A}:\, \frac{zf^{\prime }(z)}{f(z)}
	\prec h(z)\right\}.
\end{eqnarray*}

For the case $h(z)=(1+Mz)/(1+Nz) \ (-1\leq N<M\leq 1)$, the family of Janowski starlike functions $\mathcal{S}^\ast[M,N]$ is obtained (see \cite{janowski73}). When $M=1-2\delta$  $(0\leq\delta<1)$ and $N=-1$, we have the family $\mathcal{S}^\ast(\delta)$ of
starlike functions of order $\delta$. Particularly, $\delta=0$ yields the
usual class $\mathcal{S}^\ast(0)=\mathcal{S}^\ast$ of starlike functions.
Recently, Ma--Minda type families of starlike functions have been
investigated by Mendiretta \textit{et al.} \cite{Mendi}, Sok\'{o}l and
Stankiewicz \cite{b} and Cho \textit{et al.} \cite{Cho}. In \cite{Sar},
Sharma \textit{et al.} consider the subfamily $\mathcal{S}^\ast_{c}$ of
Ma--Minda classes $\mathcal{S}^*(h)$ which are endowed with the analytic
function $h_{c}(z)=1+4z/3+2z^2/3$ that is univalent, starlike with respect
to 1 and maps ${\mathbb{D}}$ onto a region bounded by the cardioid 
$(9u^2+9v^2-18u+5)^2-16(9u^2+9v^2-6u+1)=0.$
\begin{figure}[!hbt]
	\centering
	\rotatebox{0}{\scalebox{0.3}{\includegraphics{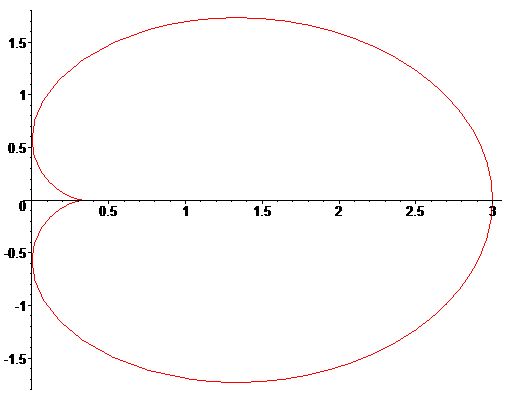}}}
	\par
	\begin{center}
		\textbf{Figure 1.} {Image domain of ${\mathbb{D}}$ under the
			function $h_{c}$}
	\end{center}
\end{figure}

In 1903, a Swedish mathematician Gosta Magnus Mittag-
Leffler  discovered a celebrated function $E_{\alpha }$  $(\alpha,z \in {\mathbb{C}},\ \Re(\alpha )>0)$ defined by
\begin{eqnarray*}
	E_{\alpha }(z)=\sum_{n=0}^{\infty }\frac{z^{n}}{\Gamma (\alpha n+1)},
\end{eqnarray*}
where $\Gamma(.)$ represents the Gamma function (see \cite{Mittag}). It is observed that the Mittag-Leffler function $E_{\alpha }$ is
an entire function of $z$ with order $[\Re(\alpha)]^{-1}$. Later, Wiman \cite{Wiman} studied the generalized Mittag-Leffler function $E_{\alpha ,\beta }$ $(\alpha ,\beta,z \in {\mathbb{C}},\ \Re(\alpha )>0,\ \Re
(\beta )>0)$ with two parameters  defined  by
\begin{eqnarray*}
	E_{\alpha ,\beta }(z)=\sum_{n=0}^{\infty }\frac{z^{n}}{\Gamma (\alpha n+\beta )}.
\end{eqnarray*}

Since then, the function $E_{\alpha ,\beta }$ and its several generalizations
arose in the solution of fractional differential equations, super diffusive transport problems, Levy flights and in some other problems. For more details, one may refer
to \cite{Bansal,Gor} and references given therein.

In \cite{Pra}, Prabhakar defined a new form of the Mittag-Leffler function with three parameters and introduced the function  $E_{\alpha ,\beta }^{\gamma }$ $\big(\alpha ,\beta ,\gamma, z \in {\mathbb{C}},\ \Re(\alpha )>0,\ \Re(\beta
)>0,\ \Re(\gamma )>0\big)$ given by
\begin{equation}
	E_{\alpha ,\beta }^{\gamma }(z)=\sum_{n=0}^{\infty }\frac{(\gamma )_{n}\
		z^{n}}{\Gamma (\alpha n+\beta )\ n!},  \label{eq:EABG}
\end{equation}
where 
\begin{eqnarray*}
	\left( \gamma \right) _{n}=\frac{\Gamma (\gamma +n)}{\Gamma (\gamma )}
	=\left\{ 
	\begin{array}{ll}
		1, & n=0 \\ 
		&  \\ 
		\gamma (\gamma +1)\ldots (\gamma +n-1),\ \ \  & n\in 
		\mathbb{N}
		
	\end{array}
	\right. 
\end{eqnarray*}
is the Pochhammer symbol. We remark that 
\begin{eqnarray*}
	E_{\alpha ,1}^{1}(z)=:E_{\alpha }(z),\ \ \ \ E_{\alpha ,\beta
	}^{1}(z)=:E_{\alpha ,\beta }(z).
\end{eqnarray*}

The function given by (\ref{eq:EABG}) is not an element of $\mathcal{A}$. Hence, we  consider the normalization of the function $E^{\gamma}_{\alpha,\beta}$ via
\begin{eqnarray*}
	\mathbb{E}^{\gamma}_{\alpha,\beta}(z)=\Gamma(\beta)zE^{\gamma}_{\alpha,
		\beta}(z)=z+\sum_{n=1}^\infty\frac{\Gamma(\beta)(\gamma)_n}{\Gamma(\alpha
		n+\beta)\ n!}z^{n+1}.
\end{eqnarray*}

Corresponding to the function $\mathbb{E}^{\gamma}_{\alpha,\beta}$ and the function $f$ of the form (\ref{aa}), consider the linear operator $\mathcal{E}^{\gamma}_{\alpha,\beta} :\mathcal{A}\rightarrow\mathcal{A}$ defined by 

\begin{eqnarray}\label{eq:sri-tom}
	\mathcal{E} _{\alpha ,\beta }^{\gamma }f(z)=\mathbb{E}_{\alpha ,\beta }^{\gamma }\ast
	f(z)=z+\sum\limits_{n=1}^{\infty }\frac{\Gamma \left( \beta \right) \left(
		\gamma \right) _{n}}{\Gamma \left( \alpha n+\beta \right) n!}a_{n+1}z^{n+1},\ \ (z\in\de)
\end{eqnarray}
where $\alpha, \beta, \gamma \in \ce$ with $\Re(\alpha)>0, \Re(\beta)>0, \Re(\gamma)>0$. This operator was introduced by Raducanu \cite{Radu} and studied for the case real-valued $\alpha, \beta, \gamma$ with  $\alpha>0, \beta>0, \gamma>0$.
\noindent From (\ref{eq:sri-tom}), it is easily verified that
\begin{equation}  \label{eq:recurrence}
	z\big(\mathcal{E}^{\gamma}_{\alpha,\beta}f(z)\big)^{\prime }=(1-\gamma)%
	\mathcal{E}^{\gamma}_{\alpha,\beta}f(z)+\gamma\mathcal{E}^{\gamma+1}_{%
		\alpha,\beta}f(z),
\end{equation}
and 
\begin{equation}  \label{eq:recurrence2}
	\alpha z\big(\mathcal{E}^{\gamma}_{\alpha,\beta+1}f(z)\big)^{\prime }=\beta%
	\mathcal{E}^{\gamma}_{\alpha,\beta}f(z)+(\alpha-\beta)\mathcal{E}%
	^{\gamma}_{\alpha,\beta+1}f(z).
\end{equation}

In view of the linear operator $\mathcal{E}^{\gamma}_{\alpha,\beta}$, we define a new function class  which maps the disc ${\mathbb{D}}$ onto a domain bounded by
cardioid.

\begin{defn}
	\label{def1.2} A function $f\in\mathcal{A}$ given by (\ref{aa}) is in the class $\mathcal{S}
	^{\gamma}_{\alpha,\beta}(h_c)$ if it satisfies the condition
	\begin{equation}  \label{eq:def1}
		\frac{z\big(\mathcal{E}^{\gamma}_{\alpha,\beta}f(z)\big)^{\prime }}{\mathcal{%
				E}^{\gamma}_{\alpha,\beta}f(z)}\prec h_c(z)=1+\frac{4z}{3}+\frac{2z^2}{3},
	\end{equation}
	where $\mathcal{E}^{\gamma}_{\alpha,\beta}$ is given by %
	(\ref{eq:sri-tom}) and $z\in{\mathbb{D}}$.
\end{defn}

This paper deals with some applications of the Briot-Bouquet differential
subordination 
\begin{equation}  \label{eq:qz}
	\phi(z)+\frac{z\phi^{\prime }(z)}{\eta \phi(z)+\mu}\prec h(z), \ \ \ (\eta,
	\mu\in{\mathbb{C}}, \eta\neq0)
\end{equation}
with $\phi(0)=h(0)=1$. If the univalent function $q(z)=1+q_1z+q_2z^2+...$ has
the feature $\phi\prec q$ for all analytic functions $\phi$, then it
is called a dominant of (\ref{eq:qz}). If $\tilde{q}\prec q$ for all dominants  $q$, then a dominant  $\tilde{q}$ is called
the best dominant. By
using a technique based upon Briot-Bouquet differential subordination which
was investigated by Miller and Mocanu \cite{Miller}, we establish several
subordination properties for the function class defined by the Mittag-Leffler
function bounded by the cardioid domain. Thus, we shall express the essential lemmas.

\begin{lem}
	\cite{Hal} \label{lem2} Assume that $h$ $(h(0)=1)$ is a convex univalent function in ${\mathbb{D}}$. Also, assume that $\phi$ of the form $\phi(z)=1+c_1z+c_2z^2+...$ $(\phi(0)=1)$
	is an analytic function in ${\mathbb{D}}$. If  
	\begin{eqnarray*}
		\phi(z)+\frac{1}{\mu}z\phi^{\prime }(z)\prec h(z), \ \ \ (\mu\neq0,\ \Re \mu\geq0)  
	\end{eqnarray*}
	then  
	\begin{equation}  \label{eq:gz}
		\phi(z)\prec \tilde{h}(z)= \frac{\mu}{z^{\mu}}\int_0^z t^{\mu-1}h(t)dt\prec
		h(z)
	\end{equation}
	and $\tilde{h}$ is the best dominant of (\ref{eq:gz}).
\end{lem}

\begin{lem}
	\cite{Miller} \label{lem3} Assume that $\eta$ ($\eta\neq0$) and $\mu$ are complex
	numbers, and $h$ $(h(0)=1)$ is a convex univalent function in ${\mathbb{D}}$ with $\Re (\eta h(z)+\mu)>0$. Also, assume that $\phi$ is analytic in ${\mathbb{D}}$ and
	satisfy (\ref{eq:qz}). If the Briot-Bouquet differential equation  
	\begin{equation}  \label{eq:q(z)}
		q(z)+\frac{zq^{\prime }(z)}{\eta q(z)+\mu}= h(z), \ \ (q(0)=1)
	\end{equation}
	has a univalent solution $q$, then  
	\begin{eqnarray*}
		\phi(z)\prec q(z)\prec h(z)
	\end{eqnarray*}
	and $q$ is the best dominant of (\ref{eq:qz}). The solution of (\ref{eq:q(z)}) is  
	\begin{equation}  \label{eq:solutionlem}
		q(z)=z^\mu [H(z)]^{\eta}\bigg(\eta\int_0^z[H(t)]^{\eta}t^{\mu-1}dt\bigg)%
		^{-1}-\mu/\eta,
	\end{equation}
	where  
	\begin{eqnarray*}
		H(z)=z\exp\int_0^z\frac{h(t)-1}{t}dt.
	\end{eqnarray*}
\end{lem}

\section{Subordination properties for the operator $\mathcal{E}^{\protect\gamma}_{\protect\alpha,\protect\beta}$}
In this section, we examine differential subordination properties for the linear
operator $\mathcal{E}^{\gamma}_{\alpha,\beta}$. Throughout this paper, we restrict our attention to the case real-valued $\alpha, \beta, \gamma$ with $\alpha>0, \beta>0$ and $\gamma>0$.

\begin{thm}
	\label{teorem1} Let $\lambda>0$ and $\zeta\geq1$. If the function $f\in\mathcal{A}$ given by (\ref{aa})  holds the condition
	\begin{equation}  \label{eq:sub}
		(1-\lambda)\frac{\mathcal{E}^{\gamma}_{\alpha,\beta}f(z)}{z}+\lambda\frac{%
			\mathcal{E}^{\gamma+1}_{\alpha,\beta}f(z)}{z}\prec 1+\frac{4z}{3}+\frac{2z^2%
		}{3},
	\end{equation}
	then 
	\begin{equation}  \label{eqRE2}
		\Re\bigg\{\bigg(\frac{\mathcal{E}^{\gamma}_{\alpha,\beta}f(z)}{z}\bigg)^{1/\zeta}\bigg\}>\bigg(\frac{\gamma}{\lambda}\int_0^1 u^{\frac{\gamma}{%
				\lambda}-1}\bigg(\frac{3-2u+u^2/2}{3}\bigg)du\bigg)^{1/\zeta}.
	\end{equation}
	The result is sharp.
\end{thm}

\begin{proof}
	Consider the  analytic function 
	\begin{eqnarray*}
		\phi(z)=\frac{\mathcal{E}^{\gamma}_{\alpha,\beta}f(z)}{z}, \ \ \ (z\in{%
			\mathbb{D}})
	\end{eqnarray*}
	with $\phi(0)=1$ in ${\mathbb{D}}$.
	Now, differentiating the above equality and using (\ref{eq:recurrence}),
	we get 
	\begin{eqnarray*}
		\frac{\mathcal{E}^{\gamma+1}_{\alpha,\beta}f(z)}{z}=\phi(z)+\frac{1}{\gamma}
		z\phi^{\prime }(z).
	\end{eqnarray*}
	By applying (\ref{eq:sub}), we write
	\begin{eqnarray*}
		(1-\lambda)\frac{\mathcal{E}^{\gamma}_{\alpha,\beta}f(z)}{z}+\lambda\frac{
			\mathcal{E}^{\gamma+1}_{\alpha,\beta}f(z)}{z}=\phi(z)+\frac{\lambda }{\gamma}
		z\phi^{\prime }(z)\prec 1+\frac{4z}{3}+\frac{2z^2}{3}.
	\end{eqnarray*}
	
	By using Lemma \ref{lem2} on the right-hand side of the above equation, we obtain
	\begin{eqnarray*}
		\phi(z)\prec\frac{\gamma}{\lambda}z^{-\gamma/\lambda}\int_0^z t^{\frac{\gamma
			}{\lambda}-1}\bigg(\frac{3+4t+2t^2}{3}\bigg)dt, 
	\end{eqnarray*}
	or
	\begin{eqnarray}\label{reel}
		\frac{\mathcal{E}^{\gamma}_{\alpha,\beta}f(z)}{z}=\frac{\gamma}{\lambda}\int_0^1 u^{\frac{\gamma}{\lambda}-1}\bigg(\frac{3+4u\vartheta(z)+2(u
			\vartheta(z))^2}{3}\bigg)du, 
	\end{eqnarray}
	where  $\vartheta$ is a Schwarz function. In \cite{Sar}, Sharma \textit{et al.} observed that 
	\begin{eqnarray*}
		\min_{|z|=r}\Re (h_{c}(z))=\min_{|z|=r}|h_{c}(z)|=h_{c}(-r)=1-\frac{4r}{3}+%
		\frac{2r^2}{3}
	\end{eqnarray*}
	if $r\leq1/2$, and also $\Re(h_{c}(\vartheta(z)))>h_{c}(-r)$. Therefore, using above relation and letting $r\rightarrow (1/2)^-$ in equality (\ref{reel}), we arrive at
	\begin{equation}\label{eqRE}
		\Re\bigg(\frac{\mathcal{E}^{\gamma}_{\alpha,\beta}f(z)}{z}\bigg)>\frac{\gamma%
		}{\lambda}\int_0^1 u^{\frac{\gamma}{\lambda }-1}\bigg(\frac{3-2u+u^2/2}{3}%
		\bigg)du>0,\ \ \ (z\in{\mathbb{D}})
	\end{equation}
	where  $\lambda>0$. Since $\Re(\vartheta^{1/\zeta})\geq \Re(\vartheta)^{1/\zeta}$ for $\Re 
	(\vartheta)>0$ and $\zeta\geq1$, from (\ref{eqRE}) we prove the inequality %
	(\ref{eqRE2}).
	To prove sharpness, we take $f\in\mathcal{A}$ defined by
	\begin{eqnarray*}
		\frac{\mathcal{E}^{\gamma}_{\alpha,\beta}f(z)}{z}=\frac{\gamma}{\lambda}%
		\int_0^1 u^{\frac{\gamma}{\lambda }-1}\bigg(\frac{3+4uz+2(uz)^2}{3}\bigg)du.
	\end{eqnarray*}
	For this function we find that 
	\begin{eqnarray*}
		(1-\lambda)\frac{\mathcal{E}^{\gamma}_{\alpha,\beta}f(z)}{z}+\lambda\frac{%
			\mathcal{E}^{\gamma+1}_{\alpha,\beta}f(z)}{z}=\frac{3+4z+2z^2}{3}
	\end{eqnarray*}
	and 
	\begin{eqnarray*}
		\frac{\mathcal{E}^{\gamma}_{\alpha,\beta}f(z)}{z}\rightarrow \frac{\gamma}{%
			\lambda}\int_0^1 u^{\frac{\gamma}{\lambda}-1}\bigg(\frac{3-2u+u^2/2}{3}\bigg)%
		du
	\end{eqnarray*}
	as $z\rightarrow (1/2)^-$. Thus, the proof is completed.
\end{proof}

By using the recurrence formula given by (\ref{eq:recurrence2}), we get the
following subordination result.

\begin{thm}
	\label{thm1} Let $\lambda>0$ and $\zeta\geq1$. If the function $f\in\mathcal{A}$ given by (\ref{aa}) holds  the condition
	\begin{equation}  \label{eq:sub2}
		(1-\lambda)\frac{\mathcal{E}^{\gamma}_{\alpha,\beta+1}f(z)}{z}+\lambda\frac{%
			\mathcal{E}^{\gamma}_{\alpha,\beta}f(z)}{z}\prec 1+\frac{4z}{3}+\frac{2z^2}{3%
		},
	\end{equation}
	then  
	\begin{equation}  \label{eqRE222}
		\Re\bigg\{\bigg(\frac{\mathcal{E}^{\gamma}_{\alpha,\beta+1}f(z)}{z}\bigg)^{1/\zeta}\bigg\}>\bigg(\frac{\beta}{\lambda\alpha}\int_0^1 u^{\frac{\beta}{\lambda\alpha}-1}\bigg(\frac{3-2u+u^2/2}{3}\bigg)du\bigg)^{1/\zeta}.
	\end{equation}
	The result is sharp.
\end{thm}

\begin{proof}
	Consider the function
	\begin{eqnarray*}
		\phi(z)=\frac{\mathcal{E}^{\gamma}_{\alpha,\beta+1}f(z)}{z}, \ \ \ (z\in{\mathbb{D}})
	\end{eqnarray*}
	is analytic with $\phi(0)=1$ in ${\mathbb{D}}$. Differentiating both sides and using (\ref{eq:recurrence2}), we find
	\begin{eqnarray*}
		\frac{\mathcal{E}^{\gamma}_{\alpha,\beta}f(z)}{z}=\phi(z)+\frac{\alpha}{\beta%
		}z\phi^{\prime }(z).
	\end{eqnarray*}
	By using (\ref{eq:sub2}), we conclude that  
	\begin{eqnarray*}
		(1-\lambda)\frac{\mathcal{E}^{\gamma}_{\alpha,\beta+1}f(z)}{z}+\lambda\frac{\mathcal{E}^{\gamma}_{\alpha,\beta}f(z)}{z}=\phi(z)+\frac{\lambda \alpha}{\beta}z\phi^{\prime }(z)\prec 1+\frac{4z}{3}+\frac{2z^2}{3}.
	\end{eqnarray*}
	
	According to the Lemma \ref{lem2}, we arrive at  
	\begin{eqnarray*}
		\phi(z)\prec\frac{\beta}{\lambda\alpha}z^{-\beta/\lambda\alpha}\int_0^z t^{%
			\frac{\beta}{\lambda\alpha}-1}\bigg(\frac{3+4t+2t^2}{3}\bigg)dt,
	\end{eqnarray*}
	and there exists a function $\vartheta\in\Omega$ such that the equality  
	\begin{eqnarray*}
		\frac{\mathcal{E}^{\gamma}_{\alpha,\beta+1}f(z)}{z}=\frac{\beta}{%
			\lambda\alpha}\int_0^1 u^{\frac{\beta}{\lambda\alpha}-1}\bigg(\frac{%
			3+4u\vartheta(z)+2(u\vartheta(z))^2}{3}\bigg)du
	\end{eqnarray*}
	holds. For $\lambda>0$, and  letting $r\rightarrow (1/2)^-$ we get  
	\begin{equation}  \label{eqRE22}
		\Re\bigg(\frac{\mathcal{E}^{\gamma}_{\alpha,\beta+1}f(z)}{z}\bigg)>\frac{%
			\beta}{\lambda\alpha}\int_0^1 u^{\frac{\beta}{\lambda\alpha }-1}\bigg(\frac{%
			3-2u+u^2/2}{3}\bigg)du>0,\ \ \ (z\in{\mathbb{D}}).
	\end{equation}
	Since $\Re(\vartheta^{1/\zeta})\geq \Re(\vartheta)^{1/\zeta}$ for $\Re 
	(\vartheta)>0$ and $\zeta\geq1$, from (\ref{eqRE22}) we prove the inequality 
	(\ref{eqRE222}).
	The sharp function is  
	\begin{eqnarray*}
		(1-\lambda)\frac{\mathcal{E}^{\gamma}_{\alpha,\beta+1}f(z)}{z}+\lambda\frac{
			\mathcal{E}^{\gamma}_{\alpha,\beta}f(z)}{z}=\frac{3+4z+2z^2}{3}
	\end{eqnarray*}
	and  
	\begin{eqnarray*}
		\frac{\mathcal{E}^{\gamma}_{\alpha,\beta+1}f(z)}{z}\rightarrow \frac{\beta}{
			\lambda\alpha}\int_0^1 u^{\frac{\beta}{\lambda\alpha}-1}\bigg(\frac{3-2u+u^2/2}{3}\bigg)du
	\end{eqnarray*}
	as $z\rightarrow (1/2)^-$. Thus, the proof is completed.
\end{proof}

Now, for a function $f$ in $\mathcal{A}$ defined by (\ref{aa}) we consider a generalized 
operator known as Bernardi-Libera-Livingston integral operator $\mathcal{L}_{\sigma}:\mathcal{A}\rightarrow \mathcal{A}$  given by 
\begin{equation}  \label{eq:bernardi}
	\mathcal{L}_{\sigma}f(z)=\frac{\sigma+1}{z^\sigma}\int_0^z
	t^{\sigma-1}f(t)dt, \ \ \ (\sigma>-1).
\end{equation}
From this operator, we easily get 
\begin{equation}  \label{eq:int}
	z(\mathcal{E}^{\gamma}_{\alpha,\beta}\mathcal{L}_{\sigma}f(z))^{\prime
	}=(\sigma+1)\mathcal{E}^{\gamma}_{\alpha,\beta}f(z)-\sigma \mathcal{E}%
	^{\gamma}_{\alpha,\beta}\mathcal{L}_{\sigma}f(z).
\end{equation}

Next, we obtain the differential subordination for the operator $\mathcal{E}^{\gamma}_{\alpha,\beta}$ associated with the Bernardi–Libera–Livingston integral operator $\mathcal{L}_{\sigma}$.
\begin{thm}
	Let $0<\lambda<1$ and $\zeta\geq1$.  If the function $f\in\mathcal{A}$ given by (\ref{aa})  satisfies the  subordination condition
	\begin{equation}  \label{eq:subo}
		(1-\lambda)\frac{\mathcal{E}^{\gamma}_{\alpha,\beta}f(z)}{z}+\lambda\frac{\mathcal{E}^{\gamma}_{\alpha,\beta}\mathcal{L}_{\sigma}f(z)}{z}\prec 1+\frac{4z}{3}+\frac{2z^2}{3},
	\end{equation}
	then  
	\begin{equation}  \label{eqRE3}
		\Re\bigg\{\bigg(\frac{\mathcal{E}^{\gamma}_{\alpha,\beta}\mathcal{L}_{\sigma}f(z)}{z}\bigg)^{1/\zeta}\bigg\}>\bigg(\frac{\sigma+1}{1-\lambda }\int_0^1 u^{\frac{\sigma+1}{1-\lambda}-1}\bigg(\frac{3-2u+u^2/2}{3}\bigg)du\bigg)^{1/\zeta}.
	\end{equation}
	The result is sharp.
\end{thm}

\begin{proof}
	Let 
	\begin{eqnarray*}
		\phi(z)=\frac{\mathcal{E}^{\gamma}_{\alpha,\beta}\mathcal{L}_{\sigma}f(z)}{z},\ \ \ (z\in{\mathbb{D}}).
	\end{eqnarray*}
	Differentiating on both sides and using the relations (\ref{eq:int}) and (\ref{eq:subo}), we get  
	\begin{eqnarray*}
		(1-\lambda)\frac{\mathcal{E}^{\gamma}_{\alpha,\beta}f(z)}{z}+\lambda\frac{\mathcal{E}^{\gamma}_{\alpha,\beta}\mathcal{L}_{\sigma}f(z)}{z}=\phi(z)+
		\frac{1-\lambda}{\sigma+1}z\phi^{\prime}(z)\prec  1+\frac{4z}{3}+\frac{2z^2}{3}.
	\end{eqnarray*}
	Applying the same method as given in Theorem \ref{teorem1}, we get (\ref{eqRE3}).
\end{proof}

\section{Subordination properties for the class $\mathcal{S}^{\protect\gamma}_{\protect\alpha,\protect\beta}(h_c)$}
In this section, we  find  univalent solutions to the Briot–Bouquet differential equation, and observe that these solutions are the best dominant to the
Briot–Bouquet differential subordination for the class $\mathcal{S}^{\protect\gamma}_{\protect\alpha,\protect\beta}(h_c)$.

\begin{thm}\label{thm1}
	If the function $f$  belongs to  the class $\mathcal{S}^{\gamma}_{\alpha,\beta}(h_c)$ such that $\mathcal{E}^{\gamma}_{\alpha,\beta}f(z)\neq 0$ for all $z\in\de$,  $\Re(\gamma)\geq 1$, and
	\begin{eqnarray*}
		\Re\bigg(\frac{4z}{3}+\frac{2z^2}{3}+\gamma\bigg)>0, \ \ (z\in\de)
	\end{eqnarray*}
	then 
	\begin{equation}  \label{eq:subclass}
		\frac{z\big(\mathcal{E}^{\gamma}_{\alpha,\beta}f(z)\big)^{\prime }}{\mathcal{E}^{\gamma}_{\alpha,\beta}f(z)}\prec q(z) \prec 1+\frac{4z}{3}+\frac{2z^2}{3},
	\end{equation}
	where 
	\begin{equation}  \label{eq:q}
		q(z)=z^{\gamma}e^{(\frac{4z}{3}+\frac{z^2}{3})}\bigg(\int_0^zt^{\gamma-1}e^{(%
			\frac{4t}{3}+\frac{t^2}{3})}dt\bigg)^{-1}-\gamma+1,
	\end{equation}
	and $q$ is the best dominant of (\ref{eq:subclass}).
\end{thm}

\begin{proof}
	Consider the analytic function
	\begin{eqnarray*}
		\phi(z)=\frac{z\big(\mathcal{E}^{\gamma}_{\alpha,\beta}f(z)\big)^{\prime }}{%
			\mathcal{E}^{\gamma}_{\alpha,\beta}f(z)}, \ \ \ (z\in{\mathbb{D}})
	\end{eqnarray*}
	with  $\phi(0)=1$. By using (\ref{eq:recurrence}), we
	get 
	\begin{eqnarray*}
		\gamma\frac{\mathcal{E}^{\gamma+1}_{\alpha,\beta}f(z)}{\mathcal{E}%
			^{\gamma}_{\alpha,\beta}f(z)}=\phi(z)+\gamma-1.
	\end{eqnarray*}
	Differentiating logarithmically with respect to $z$ and multiplying by $z$,
	from above equality we derive 
	\begin{equation}  \label{eq:phi}
		\frac{z\big(\mathcal{E}^{\gamma+1}_{\alpha,\beta}f(z)\big)^{\prime }}{%
			\mathcal{E}^{\gamma+1}_{\alpha,\beta}f(z)}=\phi(z)+\frac{z\phi^{\prime }(z)}{%
			\phi(z)+\gamma-1}\prec 1+\frac{4z}{3}+\frac{2z^2}{3}.
	\end{equation}
	
	Let us consider the differential equation 
	\begin{equation}  \label{eq:subqq}
		q(z)+\frac{zq^{\prime }(z)}{q(z)+\gamma-1}=h_c(z)=1+\frac{4z}{3}+\frac{2z^2%
		}{3},
	\end{equation}
	where $q$ $(q(0)=1)$ is analytic and $h_c(z)=1+4z/3+2z^2/3$ is convex univalent with $h_c(0)=1$ in ${\mathbb{D}}$,
	and let $P(z)=\eta h_c(z)+\mu$. In view of (\ref{eq:subqq}) and Lemma \ref{lem3},
	we observe that $\eta=1$, $\mu=\gamma-1$ and
	\begin{eqnarray*}
		P(z)=\frac{4z}{3}+\frac{2z^2}{3}+\gamma.
	\end{eqnarray*}
	
	For proving $\Re(P(z))>0$, it is enough to set $z=e^{it},\ t\in[0,\pi]$
	under the condition $\Re(\gamma)\geq 1$. Furthermore, $P(z)$ and $1/P(z)$
	are convex. Hence, there is a univalent solution of the equation (\ref{eq:subqq}). To get this solution, we  apply to the Lemma \ref{lem3}. Since $h_c(z)=1+4z/3+2z^2/3$, we find 
	\begin{eqnarray*}
		\begin{array}{ll}
			H(z) & =z\exp \int\limits_{0}^{z}\frac{h_{c}(t)-1}{t}dt \\ 
			&  \\ 
			& =z\exp \int\limits_{0}^{z}\frac{1+4t/3+2t^{2}/3-1}{t}dt=ze^{\frac{4z}{3}+\frac{z^{2}}{3}}.
		\end{array}
	\end{eqnarray*}
	
	Setting this result together with $\eta=1$ and $\mu=\gamma-1$ into the
	formula (\ref{eq:solutionlem}), we obtain (\ref{eq:q}) which is the univalent solution
	of the differential equation given by (\ref{eq:subqq}). Since $\phi$ is
	analytic and satisfy (\ref{eq:phi}), then we derive
	\begin{eqnarray*}
		\phi(z)\prec q(z)\prec h_c(z)=1+\frac{4z}{3}+\frac{z^{2}}{3}
	\end{eqnarray*}
	and $q$ is the best dominant of (\ref{eq:subclass}).
\end{proof}

\begin{thm}  If the function $f$  belongs to  the class $\mathcal{S}^{\gamma}_{\alpha,\beta}(h_c)$ such that $\mathcal{E}^{\gamma}_{\alpha,\beta+1}f(z)\neq 0$ for all $z\in\de$,  $\Re(\beta/\alpha)\geq 1$, and  
	\begin{eqnarray*}
		\Re\bigg(\frac{4z}{3}+\frac{2z^2}{3}+\frac{\beta}{\alpha}\bigg)>0,\ \ (z\in\de)
	\end{eqnarray*}
	then  
	\begin{equation}  \label{eq:subclass22}
		\frac{z\big(\mathcal{E}^{\gamma}_{\alpha,\beta+1}f(z)\big)^{\prime }}{%
			\mathcal{E}^{\gamma}_{\alpha,\beta+1}f(z)}\prec q_1(z) \prec 1+\frac{4z}{3}+\frac{2z^2}{3},
	\end{equation}
	where  
	\begin{equation}  \label{eq:q22}
		q_1(z)=z^{\frac{\beta}{\alpha}}e^{(\frac{4z}{3}+\frac{z^2}{3})}\bigg(%
		\int_0^zt^{\frac{\beta}{\alpha}-1}e^{(\frac{4t}{3}+\frac{t^2}{3})}dt\bigg)%
		^{-1}-\frac{\beta}{\alpha}+1,
	\end{equation}
	and $q_1$ is the best dominant of (\ref{eq:subclass22}).
\end{thm}

\begin{proof}
	Assume that
	\begin{eqnarray*}
		\phi(z)=\frac{z\big(\mathcal{E}^{\gamma}_{\alpha,\beta+1}f(z)\big)^{\prime }%
		}{\mathcal{E}^{\gamma}_{\alpha,\beta+1}f(z)}, \ \ \ (z\in{\mathbb{D}}).
	\end{eqnarray*}
	By using (\ref{eq:recurrence2}),
	we get  
	\begin{eqnarray*}
		\frac{\beta}{\alpha}\frac{\mathcal{E}^{\gamma}_{\alpha,\beta}f(z)}{\mathcal{E%
			}^{\gamma}_{\alpha,\beta+1}f(z)}=\phi(z)+\frac{\beta}{\alpha}-1.
	\end{eqnarray*}
	Logarithmic differentiation with respect to $z$ and routine calculations
	give  
	\begin{equation}  \label{eq:phi2}
		\frac{z\big(\mathcal{E}^{\gamma}_{\alpha,\beta}f(z)\big)^{\prime }}{\mathcal{%
				E}^{\gamma}_{\alpha,\beta}f(z)}=\phi(z)+\frac{z\phi^{\prime }(z)}{\phi(z)+%
			\frac{\beta}{\alpha}-1}\prec 1+\frac{4z}{3}+\frac{2z^2}{3}.
	\end{equation}
	
	Let us consider the differential equation  
	\begin{equation}  \label{eq:subqq22}
		q_1(z)+\frac{zq^{\prime }_1(z)}{q_1(z)+\frac{\beta}{\alpha}-1}=h_c(z)=1+%
		\frac{4z}{3}+\frac{2z^2}{3},
	\end{equation}
	where $q_1$ is analytic with $q_1(0)=1$ and $h_c(z)=1+4z/3+2z^2/3$ is convex univalent with $h_c(0)=1$ in ${\mathbb{D}}$.
	From the similar technique applied in Theorem \ref{thm1} and Lemma \ref{lem3}, the differential equation given by (\ref{eq:subqq22}) has a
	univalent solution defined by (\ref{eq:q22}). Since $\phi$ holds the
	subordination in (\ref{eq:phi2}), then we conclude that $\phi(z)\prec
	q_1(z)\prec h_c(z)=1+4z/3+2z^2/3$ and $q_1$ is the best dominant of (\ref{eq:subclass22}).
\end{proof}

\begin{thm}\label{thm4} If the function $f$  belongs to  the class $\mathcal{S}^{\gamma}_{\alpha,\beta}(h_c)$ such that $\mathcal{E}^{\gamma}_{\alpha,\beta}\mathcal{L}_{\sigma}f(z)\neq 0$ for all $z\in\de$, $\Re(\sigma)\geq 0$,   and  
	\begin{eqnarray*}
		\Re\bigg(1+\frac{4z}{3}+\frac{2z^2}{3}+\sigma\bigg)>0,\ \ \ (z\in\de)
	\end{eqnarray*}
	then $\mathcal{L}_{\sigma}f\in\mathcal{S}^{\gamma}_{\alpha,\beta}(h_c)$,
	where the operator $\mathcal{L}_{\sigma}$ is given by (\ref{eq:bernardi}).
	Moreover, if $f\in\mathcal{S}^{\gamma}_{\alpha,\beta}(h_c)$, then  
	\begin{equation}  \label{eq:subclass2}
		\frac{z\big(\mathcal{E}^{\gamma}_{\alpha,\beta}\mathcal{L}_{\sigma}f(z)\big)%
			^{\prime }}{\mathcal{E}^{\gamma}_{\alpha,\beta}\mathcal{L}_{\sigma}f(z)}%
		\prec q_2(z) \prec 1+\frac{4z}{3}+\frac{2z^2}{3},
	\end{equation}
	where  
	\begin{equation}  \label{eq:QQQq}
		q_2(z)=z^{\sigma+1}e^{(\frac{4z}{3}+\frac{z^2}{3})}\bigg(\int_0^zt^{\sigma}e^{(\frac{4t}{3}+\frac{t^2}{3})}dt\bigg)^{-1}-\sigma,
	\end{equation}
	and $q_2$ is the best dominant of (\ref{eq:subclass2}).
\end{thm}

\begin{proof}
	Consider 
	\begin{eqnarray*}
		\phi(z)=\frac{z\big(\mathcal{E}^{\gamma}_{\alpha,\beta}\mathcal{L}
			_{\sigma}f(z)\big)^{\prime }}{\mathcal{E}^{\gamma}_{\alpha,\beta}\mathcal{L}
			_{\sigma}f(z)}, \ \ \ (z\in{\mathbb{D}}).
	\end{eqnarray*}
	It follows from (\ref{eq:int}) that
	\begin{eqnarray*}
		(\sigma+1)\frac{\mathcal{E}^{\gamma}_{\alpha,\beta}f(z)}{\mathcal{E}%
			^{\gamma}_{\alpha,\beta}\mathcal{L}_{\sigma}f(z)}=\phi(z)+\sigma.
	\end{eqnarray*}
	Differentiating logarithmically with respect to $z$ and multiplying by $z$,
	from above equality we find  
	\begin{equation}  \label{eq:phi3}
		\frac{z\big(\mathcal{E}^{\gamma}_{\alpha,\beta}f(z)\big)^{\prime }}{\mathcal{%
				E}^{\gamma}_{\alpha,\beta}f(z)}=\phi(z)+\frac{z\phi^{\prime }(z)}{%
			\phi(z)+\sigma}\prec 1+\frac{4z}{3}+\frac{2z^2}{3}.
	\end{equation}
	
	Let us consider the differential equation  
	\begin{equation}  \label{eq:subqqq}
		q_2(z)+\frac{zq^{\prime }_2(z)}{q_2(z)+\sigma}=h_c(z)=1+\frac{4z}{3}+\frac{%
			2z^2}{3},
	\end{equation}
	where $q_2$ is analytic with $q_2(0)=1$ and $h_c(z)=1+4z/3+2z^2/3$ is convex univalent with $h_c(0)=1$ in ${\mathbb{D}}$.
	From the similar technique applied in Theorem \ref{thm1} and Lemma 
	\ref{lem3}, the differential equation given by (\ref{eq:subqqq}) has a
	univalent solution defined by (\ref{eq:QQQq}). Since $\phi$ satisfies the
	subordination in (\ref{eq:phi3}), then we conclude that $\phi(z)\prec
	q_2(z)\prec h_c(z)=1+4z/3+2z^2/3$, and $q_2$ is the best dominant of (\ref{eq:subclass2}).
\end{proof}

\end{document}